\documentclass{ijmart}

\usepackage{amsthm}
\usepackage{amsfonts, amssymb,amsmath}
\usepackage[LaTeXeqno]{diagrams}
\usepackage[all]{xy}
\usepackage[textsize=footnotesize,color=blue!40, bordercolor=white]{todonotes}
\usepackage{soul}

\theoremstyle{plain}
\newtheorem{thm}{Theorem}[section]
\newtheorem{lemma}[thm]{Lemma}

\newtheorem{corollary}[thm]{Corollary}

\newtheorem{question}[thm]{Question}
\newtheorem{conjecture}[thm]{Conjecture}
\newtheorem*{subclaim*}{Subclaim}
\newtheorem*{claim*}{Claim}

\theoremstyle{definition}
\newtheorem{definition}[thm]{Definition}
\newtheorem{example}[thm]{Example}

\newcommand{\Alt}{\operatorname{Alt}}
\newcommand{\BD}{\operatorname{BD}}

\newcommand{\rank}{\operatorname{rank}}
\newcommand{\id}{\operatorname{id}}

\newcommand{\Mat}{\operatorname{Mat}}

\newcommand{\GL}{\operatorname{GL}}
\newcommand{\SL}{\operatorname{SL}}
\newcommand{\PSL}{\operatorname{PSL}}

\newcommand{\Aff}{\operatorname{Aff}}
\newcommand{\infin}{\operatorname{infin}}

\newcommand{\comments}[1]{}

\newcommand{\lcm}{\operatorname{lcm}}

\begin{document}
\title[Finite rank dimension groups]{The classification problem for simple unital finite rank dimension groups}\author{Paul Ellis}
\address{Paul Ellis, Department of Mathematics and Computer Science, Manhattanville College, 2900 Purchase Street, Purchase, NY, 10577}
\email{paulellis@paulellis.org}
\urladdr{paullellis.org}

\begin{abstract}
The Borel complexity of the isomorphism problem for finite-rank unital simple dimension groups increases with rank.  This implies that the isomorphism problems for the corresponding classes of Bratteli diagrams and LDA-groups also increase with rank.
\end{abstract}

%\classno[2010]{03E15, 06F20, 20E32, 20F50}

\maketitle

A dimension group is an unperforated ordered group that satisfies the Reisz Interpolation Property.  In particular, it is a torsion-free abelian group with some added structure.  Hjorth and Thomas proved that the Borel complexity of the isomorphism problem for rank $n$ torsion-free abelian groups increases as $n$ increases.
In this paper, we show how to apply these results to dimension groups and prove that the Borel complexity of the isomorphism problem for simple, unital, rank $n$ dimension groups increases as $n$ increases (Theorem \ref{maintarget}). For a detailed treatment of Borel reducibility, see \cite{kanovei}.

Dimension groups also form a complete invariant for the class of approximately finite-dimensional $C^*$-algebras (AF-algebras), as well as for limits of diagonal embeddings of finite alternating groups (LDA-groups).  The combinatorial structure that underpins this correspondence is a directed graph called a Bratteli diagram.  We will use this to prove that, for a sufficient notion of rank, the Borel complexity of the isomorphism problem for simple, unital, thick, rank $n$ LDA groups also increases as $n$ increases.

In the next section, we give formal definitions of dimension groups, Bratteli diagrams, and LDA-groups, and we explain what is meant by simple, thick, and unital in each case. We then formally state the main theorem for dimension groups.  In Section \ref{Section-Defining the SBSs}, we define the corresponding standard Borel spaces.  
%In Section \ref{Section-state spaces}, we discuss the state space of a simple dimension group.  
We prove Theorem \ref{maintarget} for low rank dimension groups in Section \ref{section - two borel reductions and two small cases}.  After some discussion of the state space of a simple dimension group in Section \ref{Section-state spaces} and of cocycles in Section \ref{section-cocycles}, we prove Theorem \ref{maintarget} for higher rank dimension groups in Section \ref{section-proof of main theorem}.
Finally, in Section \ref{section-application to LDA groups} we show how to extend Theorem \ref{maintarget} to Bratteli diagrams and LDA-groups.

The author would like to thank Simon Thomas for years of helpful guidance, and Samuel Coskey for many helpful comments.

\section{Bratteli Diagrams, Dimension Groups, and LDA Groups}

\begin{definition}\label{Bratteli diagram}
A \emph{Bratteli diagram} $(V,E,d)$ consists of a vertex set $V$, an edge set $E$, and a labelling $d:V\rightarrow \mathbb{N}\setminus\{0\}$, where $V$ and $E$ can be
written as countable disjoint unions of nonempty finite sets $V=\bigsqcup_{n\geq 0}V_n$ and
$E=\bigsqcup_{n\geq1}E_n$ such that the following conditions hold.
\begin{enumerate}
\item There exist range and source maps $r$, $s$ from $E$ to $V$ such that \mbox{$r[E_n]\subseteq V_n$} and
$s[E_n]\subseteq V_{n-1}$.  Furthermore $s^{-1}(v)\neq\emptyset$
for all $v\in V$.
\item For each $v\in V-V_0$,
\begin{equation}
d(v)\geq \sum_{r(e)=v} d(s(e)).\tag{$\dag$}
\end{equation}
\end{enumerate}
For each $n\geq1$, the edge set $E_n$ determines a corresponding \emph{incidence matrix} $M_n=(m_{u,v})$, with
rows indexed by $V_n$ and columns indexed by $V_{n-1}$, such that $$m_{u,v}=|\{e\in E_n\mid r(e)=u\text{ and } s(e)=v\}|$$ is the number of edges joining $v$ to $u$.
\end{definition}

\begin{definition}\label{Dimension Group}
A \emph{dimension group} $(A, A^+, \Gamma)$ is an abelian group $A$ together with a \emph{positive cone} $A^+\subseteq A$ and a \emph{scale} $\Gamma\subseteq A^+$ satisfying
\begin{enumerate}
\item $A^+ + A^+\subseteq A^+$,
\item $A^+ - A^+ = A$,
\item $A^+\cap (-A^+)=\{0\}$,
\noindent\noindent  we write $a\leq b$ if $b-a\in A^+$
\item If $na\in A^+$ for some $n\in \mathbb{N}^+$, then $a\in A$ ($A$ is unperforated),
\item If $a,b\leq c,d$ then there is some $e\in A$ such that $a,b\leq e\leq c,d$ (Reisz Interpolation Property),
\item If $0\leq a\leq b$ and $b\in \Gamma$, then $a\in \Gamma$,
\item Given $a,b\in \Gamma$, there exists $c\in\Gamma$ so that $a,b\leq c$,
\item For any $a\in A^+$, there are $a_1,\ldots, a_n\in\Gamma$ so that $a=a_1+\ldots +a_n$.

\end{enumerate}
\end{definition}

\noindent (1), (2), (3) make an \emph{ordered group}. We will use the following properties of unperforated order groups:
\begin{itemize}
\item If $a\leq b$ and $c\leq d$, then $a+c\leq b+d$
\item If $na\leq nb$ for some $n\in\mathbb{N}^+$, then $a\leq b$.
\end{itemize}
\begin{example} Suppose $(a_1,\ldots,a_n)\in\mathbb{Z}^n$.  Then $(\mathbb{Z}^n, (\mathbb{Z}^n)^+,\Gamma(a_1,\ldots,a_n) )$
 is a dimension group, where\begin{itemize}
\item $(\mathbb{Z}^n)^+ = \{(z_1,\ldots,z_n)\in\mathbb{Z}^n\mid 0\leq z_i\text{ for all }1\leq i\leq n \}$,
\item For each $1\leq i\leq n$, $a_i\in \mathbb{Z}^+\cup\{\infty\}$
\item $\Gamma(a_1,\ldots,a_n) = \{(z_1,\ldots,z_n)\in (\mathbb{Z}^n)^+ \mid 0\leq z_i < a_i\text{ for all }1\leq i\leq n \}$.
\end{itemize}
\end{example}

To each Bratteli diagram $(V,E,d)$, we associate a dimension group $K_0(V,E,d)$ as follows. For each level of the vertex set $V_n=\{a_1,\ldots,a_{k_n}\}$, we associate the dimension group $(\mathbb{Z}^{k_n}, (\mathbb{Z}^{k_n})^+,\Gamma_n )$, where $\Gamma_n=\Gamma(d(a_1),\ldots,d(a_{k_{n}}))$. Then for each $n\geq 1$, let $\varphi_n:V_n\to V_{n+1}$ be the homomorphism given by matrix multiplication by the incidence matrix $M_n$ from Definition \ref{Bratteli diagram}. Since all the entries of $M_n$ are nonnegative, we see that $\varphi_n[(\mathbb{Z}^{k_{n-1}})^+]\subseteq (\mathbb{Z}^{k_n})^+$.  Then $(\dag)$ ensures that \mbox{$\varphi[\Gamma_{n-1}]\subseteq \Gamma_n$}.  Finally, $K_0(V,E,d)$ is the direct limit of this sequence of dimension groups.  A truly remarkable result of Effros, Handelman, and Shen states that every scaled dimension group can be constructed from a Bratteli diagram in this manner:
\begin{thm}\cite{effros-handelman-shen}\cite[Theorem 7.1]{effros}\label{effros-handelman-shen}
Every scaled dimension group can be expressed as $K_0(B)$ for some Bratteli diagram $B$.
\end{thm}

For each Bratteli diagram $(V,E,d)$, we shall now define a countable locally finite group
$G(V,E,d)=\bigcup_{n\geq 0}G_n$  in such a way that
\begin{enumerate}
\item Each $G_n$ is the direct product  $\Alt(\Omega_{n,1})\times...\times \Alt(\Omega_{n,k_n})$ of alternating groups on finite sets where the finite sets are indexed by the set of vertices $V_n=\{v_{n,i}\mid 1\leq i\leq k_n\}$; and
\item the number of nontrivial orbits that the subgroup $\Alt(\Omega_{n,i})$ has on $\Omega_{n+1,j}$ is exactly the number of edges from $v_{n,i}$ to $v_{n+1,j}$.
%$$m_{v_{n+1,j},v_{n,i}}=|\{e\in E_{n+1}\mid r(e)=v_{n+1,j}\text{ and }s(e)=v_{n,i}\}|$$
\end{enumerate}

Following \cite{lavrenyuk-nekrashevych}, we first choose a sequence of disjoint sets $\tilde{V_n}$ satisfying $|\tilde{V}_n|=\sum_{v\in V_n} d(v_n)$.  Then we choose corresponding surjections $\pi_n :\tilde{V}_n \to V_n$ satisfying $|\pi_n^{-1}(v)|=d(v)$ for all $v\in V_n$.  Next let $G_n =\prod_{v\in V_n} \Alt_{d(v)}$, where each factor $\Alt_{d(v)}$ acts naturally on the set $\pi_n^{-1}(v)$.  %Notice that the set $V_n$ corresponds to the set of orbits $G_n$.

Suppose we also let $\tilde{E}_{n+1}\subseteq \tilde{V}_n \times E_{n+1}$ denote the set of pairs $(\tilde{v},e)$ satisfying $s(e)=\pi_n(v)$.  Then $G_n$ acts on $\tilde{E}_{n+1}$ via the natural action on the first coordinate.  By $(\dag)$, we can find injections $\delta_{n+1}:\tilde{E}_{n+1}\to \tilde{V}_{n+1}$ satisfying $r(e)=\pi_{n+1}(\tilde{u})$ whenever $\delta_{n+1}(\tilde{v},e)=\tilde{u}$. Then since $s^{-1}(v)\neq\emptyset$ for all $v\in V$, $\delta_{n+1}$ induces an embedding $G_n\to G_{n+1}$.  That is, we identify $\tilde{E}_{n+1}$ with $\delta_{n+1}(\tilde{E}_{n+1})\subseteq \tilde{V}_{n+1}$.  Then a permutation in $G_n$ acts as it should on $\delta_{n+1}(\tilde{E}_{n+1})$ and trivially on the rest of $\tilde{V}_{n+1}$.

\begin{definition}
A group which results from a construction of this form is called an \emph{LDA-group}.
\end{definition}
%The terminology is from \cite{Lienen Puglisi}, and refers to the fact that these are \textbf{L}imits of \textbf{D}iagonal embeddings of \textbf{A}lternating groups. One goal of this paper is to examine the classification problem for thick simple unital LDA-groups.

A Bratteli diagram is called \emph{unital} if for all but finitely many $v\in V-V_0$, $$d(v)= \sum_{r(e)=v} d(s(e)).$$
An LDA-group is then called \emph{unital} if there exists a unital Bratteli diagram which gives rise to it.  Unital Bratteli diagrams also define unital dimension groups:
\begin{definition}

\begin{itemize}
\item In a dimension group $(A, A^+, \Gamma)$, an \emph{order unit} $u\in A^+$ satisfies $$a\in A^+ \implies (\exists n\in\mathbb{N})(a\leq nu).$$
\item A dimension group $(A, A^+, \Gamma)$ is \emph{unital} if there is $u\in A^+$ such that $$\Gamma=\{a\in A^+\mid a\leq  u\}.$$
In this case, we call $u$ \emph{the distinguished order unit} of $A$, and we write $(A, A^+, u)$ for  $(A, A^+, \Gamma)$.
\end{itemize}\end{definition}
\noindent In particular, if $(V,E,d)$ is a unital Bratteli diagram, then we define the distinguished order unit of $K_0 (V,E,d)$ to be
$$u=\lim_{n\to\infty} (d(a_1),\ldots,d(a_{k_{n}}))$$

Suppose we have a Bratteli diagram $(V,E,d)$, where $V=\bigsqcup_{n\geq 0}V_n$. If we take an increasing sequence of natural numbers $\{a_k\}$, then we can define a new Bratteli diagram $(V',E',d')$ on the vertex set $V'=\bigsqcup_{k\geq 0}V_{a_k}$, where $E'_k$ is determined by the incidence matrix $M'_{k_i}=M_{k_i}\cdot M_{k_i -1}\cdot \ldots \cdot M_{k_{i-1}+1}$, and $d'=d\upharpoonright_{V'}$. We then let $\sim$ be the transitive closure of this telescoping operation together with isomorphism. It is easy to see that $(V',E',d')$ defines a subsequence of the groups which define $G(V,E,d)$, and that the new embeddings are just the corresponding compositions of the original embeddings. Therefore $G(V',E',d')\cong G(V,E,d)$.  Similarly, $K_0(V',E',d')\cong K_0(V,E,d)$.  In fact, the converse holds as well for the unital case.

\begin{thm}\cite{elliott1}\label{Bratelli<DG}
If $(V,E,d)$ and $(V',E',d')$ are unital Bratteli diagrams, then $(V,E,d)\sim(V',E',d')$ if and only if $K_0(V,E,d)\cong^+ K_0(V',E',d')$.
\end{thm}

\begin{definition}
A Bratteli diagram $(V,E,d)$ is \emph{thick} if for every vertex $v\in V_i$, there is some $j>i$ and a vertex $w\in V_j$ such that there are at least two distinct paths connecting $v$ and $w$.
\end{definition}

\begin{thm}\label{laverenyuk-nekrashevych main isomorphism}\cite{lavrenyuk-nekrashevych} Suppose $B_1$ and $B_2$ are thick Bratteli diagrams.  Then $$G(B_1)\cong G(B_2) \Longleftrightarrow K_0(B_1)\cong K_0(B_2)$$
\end{thm}

\begin{corollary}\label{all three equivalent for thick unital}
If $B_1$ and $B_2$ are thick unital Bratteli diagrams, then the following are equivalent:
\begin{itemize}
\item $B_1\sim B_2$
\item $G(B_1)\cong G(B_2)$
\item $K_0(B_1)\cong K_0(B_2)$
\end{itemize}
\end{corollary}

\begin{definition}
Given a Bratteli diagram $(V,E,d)$, an \emph{ideal} is a subset $V^* \subseteq V$ such that whenever $e \in E_n$ and
$s(e)\in V^*$, then $r(e)\in V^*$.  An ideal $V^*$ is then said to be \emph{proper} if it is disjoint from some infinite path. A Bratteli Diagram is said to be \emph{simple} if it has no nonempty proper ideals.
\end{definition}

\begin{definition}If $(A,A^+,\Gamma)$ is a dimension group, then a subgroup $J\subseteq A$ is an \emph{ideal} if $J=J^+-J^+$ (where $J^+=J\cap A^+$) and $0\leq a\leq b\in J$ implies $a\in J$.  We then define a dimension group $(A,A^+,\Gamma)$ to be \emph{simple} if $\{0\}$ and $A$ are the only ideals.
\end{definition}
\noindent In particular, if $b\in A^+$ and we let
$$[b]=\{a\in A\mid 0\leq a\leq nb\text{ for some }n\in\mathbb{N}\}$$
then $J=[b]-[b]$ is the smallest ideal containing $b$.
So if $A$ is a simple dimension group, then $a\in A^+ \setminus\{0\}$ implies the ideal $[a]-[a] = A$. But $u\in A^+$ is an order unit precisely when $[u]=A^+$.  Hence if $A$ is simple, then every positive element is an order unit. Conversely, if every $a\in A^+\setminus\{0\}$ is an order unit, then $A$ clearly has non nontrivial ideals.  We restrict our attention to simple dimension groups in this paper precisely because their geometry is well-understood.  Also, simplicity is essentially the same condition for all of our objects of study:

\begin{thm}\cite{davidson}\label{B simple iff K_0(B) simple}
A Bratteli diagram $B$ is simple if and only if $K_0(B)$ is simple.
\end{thm}

\begin{thm}\cite[Theorem 5.1]{lavrenyuk-nekrashevych}\label{B simple iff G(B) simple}
If $G$ is a simple LDA-group, then it is isomorphic to $G(B)$ for some simple Bratteli diagram $B$.  Also, if $B$ is a simple thick Bratteli diagram, then $G(B)$ is simple.
\end{thm}

In addition to being a hypothesis for Theorems \ref{laverenyuk-nekrashevych main isomorphism} and \ref{B simple iff G(B) simple}, we note that the thickness condition is hardly restrictive.

\begin{thm}\cite[Theorem 5.1]{lavrenyuk-nekrashevych}\label{Thick LDA groups}
The only simple LDA-groups which are not isomorphic to $G(B)$ for some \emph{thick} simple Bratteli diagram $B$ are the groups $\Alt_n$ for \mbox{$n\in\mathbb{N}\cup\{\infty\}$}.
\end{thm}
Hence we call a simple LDA-group \emph{thick} if it is not isomorphic to $\Alt_n$ for $n\in\mathbb{N}\cup\{\infty\}$.  We have a similar situation for simple dimension groups:

\begin{definition}
The \emph{rank} of a dimension group $(A,A^+,\Gamma)$ is the size of the largest linearly independent subset.
\end{definition}

\noindent Here we mean linear independence over $\mathbb{Z}$. Note also that since dimension groups are torsion-free, then if $(A,A^+,\Gamma)$ has rank $n$, then $A$ embeds in $\mathbb{Q}^n$.

\begin{lemma}\label{rank greater than 1 is thick}
If $B$ is a Bratteli diagram such that $K_0(B)$ is simple and has rank greater than $1$, then $B$ is thick.
\end{lemma}

\begin{proof}
Suppose $x,y\in K_0(B)$ are two linearly independent elements.  Then there must be some $n>0$ so that $x,y\in \mathbb{Z}^{k_n}$.  Then since $x$ and $y$ are linearly independent, there are vertices $v,v'\in V_n$ and corresponding disjoint paths \mbox{$P=(v,e_1,e_2,\ldots)$} and $P'=(v',e'_1,e'_2,\ldots)$. Since $K_0(B)$ is simple, $B$ must be simple.  Hence for any $z\in V_i$, there is some path from $z$ to a vertex $w\in P$ and another path from $z$ to  a vertex $w'\in P'$.  Then there is a path from $w'$ to a vertex $w''\in P$.  Since $P$ and $P'$ are disjoint, this gives two distinct paths from $z$ to $w''$. Therefore $B$ is thick.

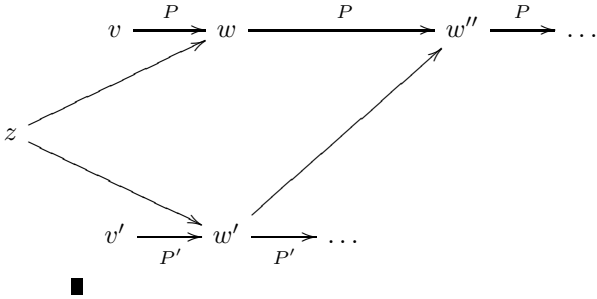
\xymatrix{
&v \ar[r]^P & w \ar[rr]^P && w'' \ar[r]^P &\ldots\\
z \ar[urr] \ar[drr]\\
&v' \ar[r]_{P'} & w' \ar[r]_{P'} \ar[uurr]&\ldots\\
}
\end{proof}

\begin{definition}
Let $n\geq 1$ and consider the standard Borel space \newline\mbox{$R(\mathbb{Q}^n)\times \mathcal{P}(\mathbb{Q}^n)\times\mathbb{Q}^n$,} where $R(\mathbb{Q}^n)$ denotes the set of full-rank subgroups of $\mathbb{Q}^n$, and $\mathcal{P}(\mathbb{Q}^n)$ denotes the power set of $\mathbb{Q}^n$.  Let $\mathcal{SDG}_n$ denote the Borel subset of $R(\mathbb{Q}^n)\times \mathcal{P}(\mathbb{Q}^n)\times\mathbb{Q}^n$ given by those
$(A,A^+,u)$ which are simple unital dimension groups (of rank $n$).  (Here we see that simplicity may be encoded by an $\mathcal{L}_{\omega_1 \omega}$-sentence, since it is equivalent to asserting that every non-zero element of $A^+$ is an order unit.)  Let $\cong_{n}^+$
denote the isomorphism relation on $\mathcal{SDG}_n$.  This is also the orbit equivalence relation given by the diagonal
action of $GL_n(\mathbb{Q})$ on $\mathcal{SDG}_n\subseteq R(\mathbb{Q}^n)\times P(\mathbb{Q}^n)\times \mathbb{Q}^n$.
\end{definition}

Our main theorem is then the following.
\begin{thm}\label{maintarget}
For all $n\geq 1$, $(\cong_{n}^+) <_{B} (\cong_{n+1}^+)$
\end{thm}

\section{The standard Borel spaces of LDA groups and Bratteli diagrams}\label{Section-Defining the SBSs}

In this section, we encode the class of thick simple unital LDA-groups as elements of a suitably chosen standard Borel space. We also encode the class of finite rank simple unital dimension groups.

\begin{definition}\label{diagonal type}
Let $G$ be a countable locally finite group and let
$$G_0\leq G_1\leq ...\leq G_n\leq ...$$
be an increasing chain of finite groups such that $G=\bigcup_{n\in\omega}G_n$.  Suppose further that
% $G_0=A_5$ and
for each
$n\geq1$,
$$G_n=A_{n,1}\times...\times A_{n,d_n}$$
where each $A_{n,i}$ is an alternating group on a finite set $\Omega_{n,i}$.  For each $1\leq i\leq d_n$, let
$$B_{n,i}=A_{n,1}\times...\times \widehat{A_{n,i}}\times...\times A_{n,d_n}$$
where $\widehat{A_{n,i}}$ indicates that $A_{n,i}$ has been omitted from the product.
\begin{enumerate}
\item[(a)] The above chain is
said to be of \emph{diagonal type} if whenever $n<m$ and $\Sigma$ is a nontrivial orbit of $G_n$ on
some $\Omega_{m,k}$, then there exists $1\leq i\leq d_n$ such that
\begin{enumerate}
\item[(1)] $|\Sigma |=|\Omega_{n,i}|$;
\item[(2)] $A_{n,i}$ acts naturally on $\Sigma$; and
\item[(3)] $B_{n,i}$ acts trivially on $\Sigma$.
\end{enumerate}
\item[(b)] The above chain is
said to be of \emph{regular type} if whenever $n<m$, then there exists $1\leq k\leq d_m$ such that $G_n$ has at least one regular orbit on $\Omega_{m,k}$.
\end{enumerate}
Then a countable locally finite group is an LDA-group precisely when it is isomorphic to the union of a chain of diagonal type.
\end{definition}

Whenever we have an embedding of two finite products of alternating groups which satisfies (a), we say that the embedding is \emph{diagonal}.
To understand (b), recall that a permutation group $H\leq \Alt(\Omega)$ is said to act \emph{regularly} if $H$ acts transitively on $\Omega$, and
\begin{itemize}
\item If $h\in H$, $x\in \Omega$, and $hx=x$, then $h=1$.
\end{itemize}
In particular, given a diagonal embedding of finite products  of alternating groups
$$\Alt(\Omega_{i,1})\times \Alt(\Omega_{i,2})\times\ldots\times \Alt(\Omega_{i,d_i})\hookrightarrow \Alt(\Omega_{j,1})\times \Alt(\Omega_{j,2})\times\ldots\times \Alt(\Omega_{j,d_j}),$$
then  $\Alt(\Omega_{i,1})\times \Alt(\Omega_{i,2})\times\ldots\times \Alt(\Omega_{i,d_i})$ cannot have any regular orbits on $\bigsqcup_{k=1}^{d_j} \Omega_{j,k}$.
The next theorem shows that when studying simple locally finite groups which can be expressed as the unions of chains of finite groups, each of which is the direct product of alternating groups, we may restrict our attention to chains of either diagonal type or regular type.

\begin{thm}\cite{Hartley-Zalesskii}\label{diagonal or regular}
Let $G$ be a countably infinite simple locally finite group, and suppose that $G$ is the union of an increasing chain
$$G_0\leq G_1\leq\ldots \leq G_n\leq\ldots$$
of finite groups, each of which is a direct product of alternating groups.  Then there exists a subsequence $\{i_n\mid n\in\omega\}$ such that the chain
$$G_{i_0}\leq G_{i_1}\leq\ldots \leq G_{i_n}\leq\ldots$$
is either of diagonal type or of regular type.
\end{thm}

\begin{definition}
Let $\mathcal{LDA}$ be the space of countable thick simple unital LDA-groups, and denote the isomorphism relation on $\mathcal{LDA}$ by $\cong_{\mathcal{LDA}}$.
\end{definition}

The next result shows that the class of countably infinite thick simple unital LDA-groups can be axiomatized by an $\mathcal{L}_{\omega_1\omega}$-sentence, and thus is a standard Borel space.

\begin{thm}
A countably infinite simple locally finite group $G$ is a thick unital LDA-group if and only if the following conditions are satisfied.
\begin{enumerate}
\item[(a)] There exists a finite subgroup $G_0$ such that every finite subset $X$ of $G$ is contained in a finite subgroup
    $$G_0\cup X\subseteq A_1\times\ldots\times A_n<G,$$
    where each $A_i$ is an alternating group on a finite set $\Omega_i$ and each element of $\sqcup \Omega_i$ lies in some nontrivial $G_0$-orbit.
\item[(b)] There exists a finite subgroup $F$ of $G$ such that whenever
$$F\leq A_1 \times\ldots\times A_n <G$$
where each $A_i$ is an alternating group on a finite set $\Omega_i$, then $F$ has no regular orbits on any of the $\Omega_i$.
\item[(c)] $G$ is not isomorphic to $\Alt_n$ for any $n\in\mathbb{N}\cup\{\infty\}$.
\end{enumerate}
\end{thm}
\begin{proof}
Assume that $G$ satisfies conditions (a), (b), and (c).  Then condition (a) allows us to express $G$ as the increasing union of finite subgroups, each of which is the direct product of alternating groups, say
$$G_0\leq G_1\leq \ldots\leq G_n\leq\ldots$$
Theorem \ref{diagonal or regular} implies that we may then select a subchain which is either of diagonal type or of regular type.  However, condition (b) implies that $G$ is not expressible as a union of a chain of regular type.  Thus $G$ must be an LDA-group.  The group $G_0$ assures us that the corresponding Bratteli diagram is unital, and hence $G$ is unital. Finally, condition (c) assures us that $G$ is thick.

Conversely, let $G$ be a thick simple unital LDA-group.  Then $G$ clearly satisfies (a), and Theorem \ref{Thick LDA groups} states that it satisfies (c).  So let $B$ be a thick simple unital Bratteli diagram satisfying $G\cong G(B)$, and let $G_0\leq G_1\leq \ldots\leq G_n\leq\ldots$ be the corresponding chain of finite subgroups of $G$.  Since $B$ is thick, there must be some $m>0$ so that one of the factors of $G_m$ is $F\cong\Alt_k$ for some $k\geq 5$.  Now suppose that
$$F\leq A_1 \times\ldots\times A_n <G$$
where each $A_i$ is an alternating group on a finite set $\Omega_i$.  Then there must be some $l>0$ such that
$$F\leq A_1 \times\ldots\times A_n <G_l.$$
Then the composed embedding of $F$ into every factor of $G_l$ is diagonal. Hence the following lemma of Zalesskii  implies that $F$ has no regular orbits on any of the $A_i$.
\begin{lemma}\cite[Lemma 10]{z}\label{zalesskii's lemma}
Suppose $m>l>k>4$.  Let $\tau_1:A(k)\rightarrow A(l)$ and $\tau_2: A(l)\rightarrow A(m)$ be embeddings of alternating groups, and let $\tau=\tau_2\circ\tau_1$.  If $\tau$ is diagonal, then both $\tau_1$ and $\tau_2$ are diagonal.
\end{lemma}\end{proof}

We encode the standard Borel space of simple thick Bratteli diagrams, as follows.  First we encode each Bratteli diagram as a member of the standard Borel space $(\mathbb{N}\times\mathbb{N})^{\mathbb{N}}$.  Fix a particular Bratteli diagram $(V,E,d)$.  We will associate to it a function $f\in(\mathbb{N}\times\mathbb{N})^{\mathbb{N}}$.  We may assume that $V=\{n\in\mathbb{N}\mid n\text{ even}\}$ and that $E=\{n\in\mathbb{N}\mid n\text{ odd}\}$.  Then encode the source and range maps by setting, for each edge $e\in E$, $f(e)=(s(e),r(e))$.  Next encode the levels of $V$ by setting, for each $v\in V_n$, $f(v)=(d(v),n)$.
\begin{definition}
Let $\mathcal{BD}$ denote the space of such functions which encode a simple thick Bratteli diagram.
\end{definition}
\begin{lemma}$\mathcal{BD}$ is a standard Borel space.
\end{lemma}

We have the following immediate application, which we will refine this theorem in Section \ref{section-application to LDA groups}.  The refinement is delayed until then, as we use not just the statement of Theorem \ref{maintarget}, but also its proof, i.e., Section \ref{section-proof of main theorem}.

\begin{thm}\label{LDA is more complex than all SDG_n}
For all $n\geq 2$, $(\cong_{n}^+) <_{B} (\cong_{\mathcal{LDA}})$.
\end{thm}
\begin{proof}
Examining the proof of Theorem \ref{effros-handelman-shen}, it is apparent that we can define a corresponding Borel function from $\mathcal{SDG}_n$ to $\mathcal{BD}$.
Then the construction of $G(V,E,d)$ defines a Borel function from $\mathcal{BD}$ to $\mathcal{LDA}$.  Lemma \ref{rank greater than 1 is thick} assures us that the Bratelli diagrams in question are thick.  Thus Theorems \ref{B simple iff K_0(B) simple} and \ref{B simple iff G(B) simple} ensures that the LDA-groups in question are simple. Finally, Theorem \ref{laverenyuk-nekrashevych main isomorphism} tells us that this is a Borel reduction.
\end{proof}

\section{Two Borel reductions and two small cases.}\label{section - two borel reductions and two small cases}
We first construct two useful Borel reductions, where the latter is the easy half of Theorem \ref{maintarget}.

\begin{thm}\label{RQn to SDGn+1}
Let $n\geq 1$.  The map $g_n:R(\mathbb{Q}^n)\rightarrow SDG_{n+1}$ given by \mbox{$g_n(G)=(A,A^+,u_A)$} where
    \begin{enumerate}
    \item $A=G\oplus\mathbb{Q}$
    \item $A^{+}=\{ (h,q)\in G\oplus\mathbb{Q} : q>0\}\cup \{ (0,0)\}$
    \item $u_A = (0,1)$
    \end{enumerate}
is a Borel reduction from $\cong_n$ to $\cong_{n+1}^+$.
\end{thm}
\begin{proof}
It is easily checked that $g_n$ maps each group $G$ to a dimension group.  To see that $(A,A^+,u_A)$ is simple, let $J\subseteq A$ be a nontrivial ideal, fix some $(g,q)\in J^+\backslash\{0\}$ and choose any $(h,r)\in A^+$.  Now choose $n\in\mathbb{N}$ so that $nq>r$.  Then since we have $n(g,q)\geq(h,r)\geq 0$, it must be the case that $(h,r)\in J^+$.  Thus $J^+=A^+$, and so $J=A$.

It is clear that $G\cong H$ implies
$g_n(G)\cong g_n(H)$.  Conversely, if $g_n(G)\cong g_n(H)$, then
$$G\cong G\oplus\{0\}=\infin(g_n(G))\cong \infin(g_n(H))=H\oplus\{0\} \cong H,$$
where $\infin(( A,A^{+},u))$ denotes the group of infinitesimals of $( A,A^{+},u)$.
\end{proof}

\begin{thm}\label{SDGn to SDGn+1}
Let $n\geq 2$.  The map $f_n:SDG_n \rightarrow SDG_{n+1}$ given by $f_n((A,A^+,u_A))=(B,B^+,u_B)$ where
    \begin{enumerate}
    \item $B=A\oplus\mathbb{Q}$
    \item $B^{+}=\{ (a,q)\in A\oplus\mathbb{Q} : a\in A^+ \backslash\{0\}\text{ and }q>0\}\cup \{ (0,0)\}$
    \item $u_B=(u_A,1)$
    \end{enumerate}
is a Borel reduction from $\cong_n^+$ to $\cong_{n+1}^+$.
\end{thm}
\begin{proof}
We first need to check that $(B,B^+,u_B)\in SDG_{n+1}$.  It is easy to check that $(B,B^+,u_B)$ is an unperforated ordered group.  For example, to see that \mbox{$B^+ - B^+ = B$}, let $(a,q)$ be any element of $B$.  Then there are $a_1,a_2\in A^+$ so that $a=a_1 - a_2$.  If either $a_1$ or $a_2$ are $0$, replace them with $a_1 +a$ and $a_2 +a$, where $a>0$.  Next let $q_1,q_2$ be any two positive rational numbers so that $q_1-q_2=q$.  Then $(a,q)=(a_1,q_1)-(a_2,q_2)$ and $(a_1,q_1),(a_2,q_2)\in B^+$.

To see that $(B,B^+,u_B)$ satisfies the Riesz interpolation property, consider elements $(a_i,q_i),(b_j,p_j)\in B$ $(1\leq i,j\leq 2)$ such that $(a_i,q_i)\leq(b_j,p_j)$.  First note that if $q_i=p_j$ for some $i, j\in \{1,2\}$, then it must be the case that \mbox{$a_i=b_j$} and then we can choose $(a_i,q_i)$ to interpolate.  Thus we can assume that $q_1,q_2<p_1,p_2$, and so $a_1, a_2 < b_1, b_2$.  Then let $q$ be some rational number such that $q_1,q_2<q<p_1,p_2$.  Now by the following lemma, there exists $c\in A$ with $a_i< c< b_j$ for $1\leq i,j\leq 2$, and so we can choose $(c,q)$ to interpolate:

\begin{lemma}\cite[Corollary 1.2]{effros-handelman-shen}\label{simple and not Z implies strong Riesz}
If $(A,A^+)\ncong(\mathbb{Z},\mathbb{Z}^+)$ is a simple dimension group, then $(A,A^+)$ satisfies the strong Riesz interpolation property: given elements $a,b,c,d\in A$, if $a,b<c,d$, then there is some $e\in A$ so that $a,b<e<c,d$.
\end{lemma}

To see that $(B,B^+,u_B)$ is simple, let $J$ be a nontrivial ideal of $(B,B^+,u_B)$, and let $(a,q)\in J^+\backslash\{0\}$.  Now let $(b,r)$ be any other element of $B^+$. Since $(A,A^+,u_A)$ is simple and $a\in A^+\backslash\{0\}$, we know that $a$ is an order unit in $(A,A^+)$. That is, there is some natural number $n\in\mathbb{N}$ such that $na-b\in A^+\backslash\{0\}$.  Since $q>0$, there is some natural number $n'\in\mathbb{N}$  such that $n'q-r>0$.  Then letting $m=\max\{n,n'\}$, $m(a,q)-(b,r)\in B^+$, and so $(b,r)\in J^+$.  Thus $J^+ =B^+$ and so $J=B$.

We now need to check that $f_n$ is a Borel reduction.  Clearly \mbox{$(A,A^+,u_A) \cong (C,C^+,u_C)$} implies $f_n(A,A^+,u_A) \cong f_n(C,C^+,u_C)$.
On the other hand, suppose that

\noindent $f_n(A,A^+,u_A)=(B,B^+,u_B)$,  $f_n(C,C^+,u_C)=(D,D^+,u_D)$, and $\varphi\in\GL_{n+1}(\mathbb{Q})$ such that $\varphi(B,B^+,u_B)=(D,D^+,u_D)$.  Consider the set
\begin{align*}
B^\circ &=\{b\in B\mid b\notin B^+\text{ and for every }b'\in B^+\backslash\{0\}, b+b'\in B^+\}\cup\{(0,0)\}\\
&=\{(0,q)\in B\mid q\in\mathbb{Q}^+\}\cup\{(a,0)\in B\mid a\in A^+\}.
\end{align*}
Then the first equality above shows that $\varphi(B^\circ)=D^\circ$.  The second shows that either
\begin{itemize}
\item $\varphi(A^+)=C^+$, and $\varphi$ extends linearly to an isomorphism $\varphi:(A,A^+,u_A)\cong(C,C^+,u_C)$; or
\item $\varphi(A^+)=\mathbb{Q}^+$, and $\varphi$ extends linearly to an isomorphism\;\; $\varphi:(A,A^+,u_A)\cong(\mathbb{Q},\mathbb{Q}^+,\mathbb{Q})$.
\end{itemize}
However, the latter is impossible since $n\geq 2$.
\end{proof}

\begin{thm}\label{1 is less than 2}
$(\id_{2^\mathbb{N}})\sim_B (\cong_{1}^+) <_{B} (\cong_{2}^+)$.
\end{thm}

\begin{proof}
In \cite{t1}, Thomas showed that  $(\cong_1) \sim_B E_0$. Then since $(\cong_1) \leq_B (\cong^+_2)$, we know that
$\cong^+_2$ is not smooth.

On the other hand, $\cong^+_1$ is smooth, since the isomorphism class of a rank $1$ simple dimension group \mbox{$(A,A^+,u_A)\in SDG_1$} is determined by the set of prime divisors of $u_A$.
\end{proof}

\begin{thm}\label{2 is less than 3}
$(\cong_{2}^+) <_{B} (\cong_{3}^+)$.
\end{thm}
To prove this, we make use of the following fact:

\begin{thm}\cite[Section 2.4]{jkl}\label{Frechet Amenable Theorem}
Suppose $G$ and $H$ are countable groups acting on the standard Borel spaces $X$ and $Y$ respectively, and let $E_G^X$, $E_H^Y$ be the corresponding orbit equivalence relations.  Suppose $E_G^X \leq_B E_H^Y$ or $E_G^X \subseteq E_H^Y$. Suppose furthermore that there is a $G$-invariant probability measure on $X$, and the action is free on an invariant Borel set of measure $1$.  If $H$ is amenable, then $G$ is amenable.
\end{thm}

\begin{proof}[Proof of Theorem \ref{2 is less than 3}]
Theorem \ref{SDGn to SDGn+1} implies that $(\cong_{2}^+) \leq_{B} (\cong_{3}^+)$.
Instead of analyzing $\cong_{2}^+$, we first note that it is enough to analyze the Borel equivalence relation
obtained by restricting $\cong_{2}^+$ to those dimension groups $(A,A^+,u_A)$ for which
$u_A=(1,0)\in\mathbb{Q}^2$.  Denote this space of rank $2$ dimension groups by $SDG_2^{<e_0>}$ and the resulting equivalence relation by $(\cong_2^+)^{<e_0>}$.  Then we have a Borel reduction from $\cong_2^+$ to $(\cong_2^+)^{<e_0>}$ via $(A,A^+,u_A)\mapsto g(A,A^+,u_A)$, where $g$ is some element of $\GL_2(\mathbb{Q})$ such that $g(u_A)=(1,0)$.  Now we have that $(\cong_2^+)^{<e_0>}$ is the orbit equivalence relation of the natural action of the group
$$H=\left\{ \left(
\begin{array}{ll}
a & b\\
c & d
\end{array}
\right) \in GL_n(\mathbb{Q})\mid a=1,c=0\right\}$$ on $SDG_2^{(e_0)}$.  Then since $H$ is solvable, it is also amenable.

Now suppose that $(\cong_3^+)\leq_B (\cong_2^+)$, and thus \mbox{$(\cong_2) \leq_B (\cong_3^+) \leq_B (\cong_2^+) \leq_B (\cong_2^+)^{<e_0>}$}.  In \cite[Section 5]{hjorth}, Hjorth has constructed a $\PSL_2(\mathbb{Z})$-invariant measure $\mu$ on a space $Y$ that is Borel isomorphic to a subset of $R(\mathbb{Q}^2)$, together with a Borel subset $X\subset Y$ with $\mu(X)=1$ such that $\PSL_2(\mathbb{Z})={\SL_2(\mathbb{Z})}/\{1,-1\}$ acts freely on $X$.  Thus Theorem \ref{Frechet Amenable Theorem} would imply that $\PSL_2(\mathbb{Z})$ is amenable.  However, this is not the case, since $\PSL_2(\mathbb{Z})$ contains an isomorphic copy of $\mathbb{F}_2$, namely, $\left<
\left(
\begin{array}{c c}
1 & 1\\
0 & 1 \end{array}
\right),
\left(
\begin{array}{c c}
1 & 0\\
1 & 1 \end{array}
\right)
\right>/\{-1,1\}$.
\end{proof}

\section{The Geometry of Finite Rank Simple Unital Dimension Groups}\label{Section-state spaces}

In order to prove Theorem \ref{maintarget} for $n\geq 3$, we will need to understand the space of states of a simple unital dimension group of finite rank. For more detailed treatments, see Chapter 4 of \cite{effros} and Chapter 14 of \cite{goodearl2}. Fix $n\geq 3$ and some $(A,A^+,u)\in SDG_n$.  We say that a homomorphism $p:A\rightarrow\mathbb{R}$ is a \emph{state} if $p$ is positive
(i.e., $p(A^+)\geq0$), and $p(u)=1$.  We let $S_u(A,A^+)$ be the set of all states on $(A,A^+,u)$, and we give it the
weakest topology for which each of the functions $\hat{a}:f\mapsto f(a)\text{ }(a\in A)$ is continuous.

It is clear that $S_u(A,A^+)$ is convex. Since $A$ has finite rank,  $S_u(A,A^+)$ is also compact.  Let $E(S_u(A,A^+))$ be the convex hull of the extreme points of $S_u(A,A^+)$.
It is then a consequence of the Krein-Milman theorem that $E(S_u(A,A^+))$ has at most $n-1$ points.  (See also \cite[Proposition 14.21]{goodearl2})

%says that $S_u(A,A^+)$ is
%the convex hull of its extreme points.  Let $E(S_u(A,A^+))$ be this set of extreme points.  We then have:

%\begin{thm}\label{state space is finite}\cite[Proposition 14.21]{goodearl2}
%If $(A,A^+,u)$ is a simple unital dimension group with finite rank $n\geq 2$, then $E(S_u(A,A^+))$ has at most $n-1$ points.
%\end{thm}

Next, we wish to explore the manner in which $E(S_u(A,A^+))$ determines the positive cone $A^+$. We let $\Aff S_u(A,A^+)$ be the affine functions on $S_u(A,A^+)$.
Then we define a positive homomorphism \mbox{$\theta:A\rightarrow\Aff S_u(A,A^+):a\mapsto\hat{a}$} by letting $\hat{a}(p)=p(a)$.  (Note that $\hat{u}=1$.)

\begin{definition}
Suppose that $u$ is an order unit in an unperforated order group $A$.  We say that $a\in A$ is
\emph{infinitesimal} if $-u\leq na\leq u$ for all $n\in\mathbb{N}$.
\end{definition}
Notice that the infinitesimal elements do not depend on the choice of $u$, and that the set of infinitesimals forms a subgroup of $A$.

\begin{thm}\cite[Corollary 1.5]{effros-handelman-shen}\label{nature of theta}
If $(A,A^+,u)$ is a simple dimension group, then the map
$$\theta:A\rightarrow\Aff S_u(A,A^+)$$
determines the order on $A$ in the sense that $$A^+=\{a\in A:\hat{a}(p)>0\text{ for all }p\in S_u(A,A^+) \}\cup\{0\}.$$  Furthermore, we have $a\in\ker\theta$ (i.e., $\hat{a}=0$)
if and only if $a$ is infinitesimal.
\end{thm}

\begin{corollary}\label{extreme points determines A^+}
Fix $n\geq 1$ and let $(A,A^+,u), (B,B^+,v)\in SDG_n$.  Suppose that $A=B$ and $u=v$.  Then $A^+=B^+$ if and only if $S_u(A,A^+)=S_v(B,B^+)$.
\end{corollary}

\section{Cocycles}\label{section-cocycles}
Let $G$ be a locally compact second countable (lcsc) group, and let $X$ be a standard Borel $G$-space with invariant probability measure $\mu$.  Let $H$ be an lcsc group.  A \emph{cocycle} of the $G$-space $X$ into $H$ is a Borel map $\alpha: G\times X\rightarrow H$ such that for all $g,h\in G$,
$$\alpha(hg,x)=\alpha(h,g\cdot x)\alpha(g,x)\;\;\mu\text{-a.e.}(x).$$
If \mbox{$\beta: G\times X\rightarrow H$} is also a cocycle, we say that \emph{$\alpha$ is equivalent to $\beta$} if there is a Borel map $A:X\rightarrow H$ such that for all $g\in G$,
$$\alpha(g,x)=A(g\cdot x)\beta(g,x)A(x)^{-1}\;\;\mu\text{-a.e.}(x).$$

In addition to being the only type of cocycle which we will encounter in this paper, the following canonical example motivates the above definitions.  Suppose $E=E_G^X$ and $F=E_H^Z$, where $H$ acts freely on $Z$.  Let $f:X\rightarrow Z$ be a Borel function such that $xEy$ implies $f(x)Ff(y)$, i.e., $f$ is a Borel homomorphism from $E$ to $F$.  Then the function $\alpha:G\times X\rightarrow H$ defined by $f(g\cdot x)=\alpha(g,x)\cdot f(x)$ is a cocycle. (There exists a \emph{unique} such element $\alpha(g,x)\in H$ since the action of $H$ on $Z$ is free.)

There are various cocycle reduction results which say that, under certain hypotheses, cocycles $\alpha$ are equivalent to cocycles $\beta$, whose range $\beta(G\times X)$ is contained in a ``small'' subgroup of $H$.  In Section \ref{section-proof of main theorem} we shall make essential use of the following such theorem.

\begin{thm}\cite[2.3]{t2}\label{cocycletheorem}
Let $n\geq 3$ and let $X$ be a standard Borel $SL_n(\mathbb{Z})$-space with an invariant ergodic probability
measure.  Suppose that $G$ is an algebraic $\mathbb{Q}$-group such that $\dim G< n^2-1$ and that $H\leq
G(\mathbb{Q})$. Then for every Borel cocycle $\alpha\colon SL_n(\mathbb{Z})\times X\rightarrow H$, there exists
an equivalent cocycle $\beta$ such that $\beta(SL_n(\mathbb{Z})\times X)$ is contained in a finite subgroup of
$H$.
\end{thm}

\section{Proof of Theorem \ref{maintarget}}\label{section-proof of main theorem}

Two torsion-free abelian groups $A,B\in R(\mathbb{Q}^n)$ are isomorphic if and only if there is some $\varphi\in \GL_n(\mathbb{Q})$ such that $A=\varphi(B)$.  Let $\cong_n$ denote this isomorphism relation.

\begin{thm}\cite{hjorth}\cite{t2}\label{hjorth/thomas}
For all $n\geq 1$, $(\cong_n) <_B (\cong_{n+1})$.
\end{thm}

A key ingredient in Thomas' proof that $(\cong_{n-1})<_B (\cong_{n})$, for $n\geq 3$ was that given a Borel homomorphism $f:R(\mathbb{Q}^{n})\rightarrow R(\mathbb{Q}^{n-1})$ from $\cong_{n}$ to $\cong_{n-1}$, he was able to reduce the analysis to the situation where the domain of $f$ was a standard Borel $\SL_n(\mathbb{Z})$-space $X$ with an invariant ergodic probability measure $\mu$.   The construction of the following measure space can be found in Sections 3 and 4 of \cite{thomas-p-local}.

\begin{definition}
Let $\mathbb{P}$ denote the set of primes.  If $p\in\mathbb{P}$, then a group \mbox{$A\in R(\mathbb{Q}^n)$} is said to be \emph{$p$-local} iff $A=qA$ for every prime $q\neq p$; i.e., $A$ is a $\mathbb{Z}_{(p)}$-module, where
$$\mathbb{Z}_{(p)}=\left\{\frac{a}{b}\in\mathbb{Q}\mid a,b\in\mathbb{Z}\text{ and } b\text{ is relatively prime to }p\right\}.$$
Let $R^{(p)}(\mathbb{Q}^n)$ denote the $p$-local subgroups of $\mathbb{Q}^n$ of rank $n$.
\end{definition}

\begin{thm}\label{existence of probability space}
Let $n\geq 3$, and let $p$ be any prime.  Then there exists a standard Borel $\SL_n(\mathbb{Z})$-space $X_n$ with an invariant ergodic probability measure $\mu_n$.  Also there is a countable-to-one map $\sigma_n:X_n\to R(\mathbb{Q}^n)$ which is a Borel homomorphism from $E_{\SL_n(\mathbb{Z})}^{X_n}$ to $\cong_n$.  Furthermore, $\sigma_n$ may be chosen so that all the groups in the image are $p$-local.
\end{thm}

In order to prove Theorem \ref{maintarget} for the case $n\geq 3$, we will first reduce the analysis to the case of a Borel homomorphism $f:SDG_{n+1}\rightarrow SDG_{n}$ so that for any two dimension groups $(A,A^+,u_A), (B,B^+,u_B)$ in the image of $f$, there exists $\varphi\in \GL_n(\mathbb{Q})$ so that $\varphi(A)=B$ and $\varphi(\infin(A,A^+,u_A))=\infin(B,B^+,u_B)$.

Suppose \mbox{$h:R(\mathbb{Q}^{n})\to R(\mathbb{Q}^{n-1})$} is a Borel homomorphism.  Thomas showed that if $n\geq 3$, then with respect to a suitable invariant ergodic probability measure, $h$ maps a measure one subset of an $SL_n(\mathbb{Z})$-invariant Borel subset of $R(\mathbb{Q}^{n+1})$ to a single isomorphism class of $R(\mathbb{Q}^{n})$.  A key fact in this proof is $\dim(\GL_{n-1}(\mathbb{Q})) < n^2-1$.

In our case, we let $g_n:R(\mathbb{Q}^n)\rightarrow SDG_{n+1}$ be the Borel reduction from $\cong_n$ to $\cong_{n+1}^+$ defined in Theorem \ref{RQn to SDGn+1}, and let $\pi'_n :SDG_{n}\rightarrow R(\mathbb{Q}^n)\times S(\mathbb{Q}^n)$ be the forgetful map $\pi'_n (A,A^+,u)=(A,\infin(A,A^+,u_A))$.  Consider a Borel homomorphism \mbox{$f:SDG_{n+1}\rightarrow SDG_{n}$} (recall $n\geq 3$).  Then composing these maps, we obtain a Borel homomorphism $h=\pi'_{n}\circ f\circ g_n$ from $\cong_n$ to $\cong_{n^*}$, where $S(\mathbb{Q}^n)$ is the space of all subgroups of $\mathbb{Q}^n$ and $\cong_{n^*}$ is the orbit equivalence relation of the diagonal action of $\GL_{n}(\mathbb{Q})$ on $R(\mathbb{Q}^n)\times S(\mathbb{Q}^n)$.  It turns out that the only reason that this is not enough to apply Thomas' analysis is that $\dim(\GL_{n}(\mathbb{Q})) \not< n^2-1$.

To fix this, we observe that by first adjusting by an appropriate element of $\GL_n(\mathbb{Q})$, we can assume that the order unit of every dimension group in the image of $f$ is $u=e_0$,  thus shrinking the group which acts on $SDG_{n}$. In particular, let $\Mat_{n}^{\langle e_0\rangle}(\mathbb{Q})\subset \Mat_n(\mathbb{Q})$  be the subset of all $n\times n$ matrices which fix the one-dimensional subspace $\langle e_0\rangle$.  Let $\GL_{n}^{\langle e_0\rangle}(\mathbb{Q})=\GL_n(\mathbb{Q})\cap \Mat_{n}^{\langle e_0\rangle}(\mathbb{Q})$. Let $\cong_{n^*}^{\langle e_0\rangle}$ be the orbit equivalence relation of the diagonal action of $\GL_{n}^{\langle e_0\rangle}(\mathbb{Q})$ on
$R(\mathbb{Q}^n)\times S(\mathbb{Q}^n)$.  Then $\dim(\GL_{n}^{\langle e_0\rangle}(\mathbb{Q}))< n^2-1$ (for $n>3$), as desired.  This proof of this theorem is then a straightforward analogue of the proof of \cite[Theorem 3.5]{t1}

\begin{thm}\label{onegroupisomorphismclass}
Let $n\geq 3$ and let $X$ be a standard Borel $\SL_{n}(\mathbb{Z})$-space with an invariant ergodic
probability measure $\mu$. Suppose that $f\colon X\rightarrow R(\mathbb{Q}^{n})\times S(\mathbb{Q}^{n})$ is a
Borel function such that $xEy\Rightarrow f(x) \cong_{n^*}^{\langle e_0\rangle} f(y)$.  Then there exists an
$\SL_n(\mathbb{Z})$-invariant Borel subset $M\subset X$ with $\mu(M)=1$ such that $f$ maps  $M$ into a single
$\cong_{n^*}^{\langle e_0\rangle}$ -class.
\end{thm}

Now let $n\geq 3$ and assume toward a contradiction that $f:SDG_{n+1}\rightarrow SDG_{n}$ is a Borel reduction from $\cong^+_{n+1}$ to $\cong^+_{n}$.  Let \mbox{$g_n:R(\mathbb{Q}^n)\rightarrow SDG_{n+1}$} be the Borel reduction defined in Theorem \ref{RQn to SDGn+1}.  Then $h=f \circ g_n$ is a Borel reduction from
$\cong_n$ to $\cong_{n}^{+}$.  Letting $X=X_n$, $\mu=\mu_n$, and $\sigma=\sigma_n$ as in Theorem \ref{existence of probability space}, we have that
\begin{enumerate}
\item[(a)] $X$ is a standard Borel $\SL_n(\mathbb{Z})$-space with $\SL_n(\mathbb{Z})$-invariant ergodic probability measure $\mu$,
\item[(b)] $\sigma$ is a Borel homomorphism from $E^{X_n}_{\SL_n(\mathbb{Z})}$ to $\cong_n$, and
\item[(c)] $\sigma$ is countable-to-one and hence does not map a measure one subset of $X$ to a single $\cong_n$-class.
\end{enumerate}
Adjusting by the appropriate elements of $GL_n(\mathbb{Q})$, we may assume that the order unit of
every element in the range of $h$ is $u=e_0$.
Now let \mbox{$\pi :SDG_{n}\rightarrow R(\mathbb{Q}^n)\times S(\mathbb{Q}^n)$} be the map $\pi (A,A^+,u)=(A,\infin(A,A^+,u))$. Then Theorem \ref{onegroupisomorphismclass} implies that we may assume that $\pi\circ h\circ \sigma$ maps
$X$ into a single $\cong^{\langle e_0\rangle}_{n^*}$ -class. Hence, after adjusting by the appropriate elements of $\GL^{\langle e_0\rangle}_n(\mathbb{Q})$, we
can assume that $\pi\circ h\circ\sigma$ maps $X$ to a single pair, say $(A,I)$. So we have reduced our analysis to the case when  all the dimension groups in the
image of $h\circ\sigma$ have the same underlying torsion-free abelian group $A$, the same group of infinitesimals $I$, and the same distinguished order unit $u=e_0$.

In both of the following cases, we will use (a) and (b) above to show that $h\circ\sigma$ maps a measure-one subset of $X$ to a single $\cong^+_n$-class.  However, this violates (c), and thus completes the proof of Theorem \ref{maintarget}.

\subsection{Case I: $I=\{0\}$}

Fix some $x\in X$.  Let $(A,A^+_x,u)= (h\circ\sigma)(x)$, and let $S_x$ be the stabilizer of $(A,A^+_x,u)$ in
$GL^{\langle e_0\rangle}_n(\mathbb{Q})$.
\begin{claim*}
$S_x$ is finite.
\end{claim*}

\begin{proof}

We examine the action of $S_x$ on the state space $S_u(A,A^+_x)$ defined by $$\varphi.p(a)=p(\varphi^{-1}(a))\text{ for }p\in S_u(A,A^+_x)\text{ and }\varphi\in S_x.$$
Notice that $\varphi\in S_x$ implies that $\varphi^{-1}(a)\in A$ for each $a\in A$, and so the above is well-defined.  Notice also that
\begin{enumerate}
\item $\varphi.p(u)=p(\varphi^{-1}(u))=p(u)=1$; and
\item for any $a\in A^+$, $\varphi^{-1}(a)\in A^+$, and so $\varphi.p(a)=p(\varphi^{-1}(a))\in {\mathbb{R}}^+$.
\end{enumerate}
Thus $\varphi.p\in S_u(A,A^+_x)$.  Notice also that, for any $\varphi\in S_x$, $p,q\in S_u(A,A^+_x)$, and $0\leq \alpha\leq 1$,
$$\varphi.(\alpha p+(1-\alpha)q)=\alpha\varphi.p+(1-\alpha)\varphi.q.$$
Thus any $\varphi\in S_x$ is an affine permutation of the classical simplex $S_u(A,A^+_x)$, and so it must permute the elements of the finite set $E(S_u(A,A^+_x))$.  Hence the following statement implies that $S_x$ is finite.

    \begin{subclaim*}
    If $\varphi\in S_x$, and $\varphi$ acts as the identity on $E(S_u(A,A^+_x))$, then $\varphi=\id$.
    \end{subclaim*}

    \begin{proof}In this case, since each $p\in S_u(A,A^+_x)$ is an affine combination of elements of $E(S_u(A,A^+_x))$, $\varphi$ fixes \emph{every} state $p\in S_u(A,A^+_x)$.
    Now given any $a\in A_x$, recall that $\hat{a}\in \Aff(S_u(A,A^+_x))$ is defined by $\hat{a}(p)=p(a)$.  So choose any state $p\in S_u(A,A^+_x)$ and any $a\in A$.  Then $p(a)=\varphi^{-1}.p(a)$.  Thus $p(a)=p(\varphi(a))$, and so $p(a-\varphi(a))=0$.    This implies that $\widehat{a-\varphi(a)}(p)=0$, and since our choice of $p$ was arbitrary, $\widehat{a-\varphi(a)}=0$.  But since $I=\{0\}$, Theorem \ref{nature of theta} implies $a-\varphi(a)=0$.  Thus $a=\varphi(a)$ for every $a\in A$, and so $\varphi=\id$.
    \end{proof}
\end{proof}

Thus there are only countably many possibilities for the stabilizer of $(h\circ\sigma)(x)=(A,A_x,u)$ in $\GL_n^{\langle
e_0\rangle}(\mathbb{Q})$ and we can proceed as in the proof of \cite[Theorem 3.5]{t1}.  For completeness, we give the details here.  For the rest of this case, let $E=E^X_{\SL_n(\mathbb{Z})}$. Since $\mu$ is countably-additive, there exists a Borel subset $X_1\subseteq X$ with $\mu(X_1)>0$ and a fixed
finite subgroup $S$ of $\GL_n^{\langle e_0\rangle}(\mathbb{Q})$ such that $S_x=S$ for all $x\in X_1$.  By the
ergodicity of $\mu$, we have that $\mu(\SL_n(\mathbb{Z}).X_1)=1$. In order to simplify notation, we shall assume
that $\SL_n(\mathbb{Z}).X_1=X$. After slightly adjusting $h\circ\sigma$ if necessary, we can also assume that $S_x=S$ for
all $x\in X$. (That is, let $c\colon X\rightarrow X$ be a Borel function such that $c(x)E x$ and $c(x)\in X_1$
for all $x\in X$.  Then we can replace $h\circ\sigma$ with $h\circ\sigma\circ c$.)

Now suppose that $x,y\in X$ and that $xEy$.  Then $(A,A^+_x,u) \cong (A,A^+_y,u)$ and so there exists $\varphi\in \GL_n^{\langle
e_0\rangle}(\mathbb{Q})$ such that $\varphi(A,A^+_x,u)=(A,A^+_y,u)$. Notice that
$$\varphi S\varphi^{-1}=\varphi S_x\varphi^{-1}=S_y=S$$
and so $\varphi\in N=N_{\GL_n^{\langle e_0\rangle}(\mathbb{Q})}(S)$.  Let $H=N/S$ and for each $\varphi\in N$,
let $\overline{\varphi}=\varphi S$. Then we can define a cocycle $\alpha\colon \SL_n(\mathbb{Z})\times
X\rightarrow H$ by
$$\alpha(g,x)=\text{ the unique element }\overline{\varphi}\in H\text{ such that }\varphi(A,A^+_x,u)=(A,A^+_{g\cdot x},u).$$

Now since $S$ is finite, it is a closed subgroup of $N$, and so $H$ is a  algebraic $\mathbb{Q}$-group (See for example \cite[5.5.10]{Springer}). Furthermore we have the following, where the last inequality holds because $n\geq 3$,
$$\dim H\leq\dim GL_n^{\langle e_0\rangle}(\Omega)=n^2-(n-1)<n^2-1.$$

Thus, by Theorem \ref{cocycletheorem}, $\alpha$ is equivalent to a cocycle $\beta$ such that
$\beta(SL_n(\mathbb{Z})\times X)$ is contained in a finite subgroup $K$ of $H$. To simplify notation, we shall
assume that $\beta=\alpha$. Then for each $x\in X$,
\begin{eqnarray*}
\Phi(x)&=&\{\varphi(A,A^+_x,u)\mid\overline{\varphi}=\alpha(g,x)\text{ for some }g\in \SL_n(\mathbb{Z})\}\\ &=&\{(A,A^+_z,u)
\mid zEx\}
\end{eqnarray*}
is finite; and clearly if $xEy$, then $\Phi(x)=\Phi(y)$. But this means that $\Phi$ is a Borel homomorphism from $E$ to the identity relation on the standard Borel space of finite subsets of $SDG_n$.  Hence, by Theorem \ref{reformulation of ergodic}, there exists a Borel subset $X_2\subseteq X$ with $\mu(X_2)=1$ such that $\Phi(x)=\Phi(y)$ for all $x,y\in X_2$; and this means that $h\circ\sigma$ maps $X_2$ into a single $\cong_n^+$-class, as desired.

Of course, after a suitable adjustment of $h\circ\sigma$, we can assume that $h\circ\sigma$ maps $X_2$ to a single dimension group.  This observation will be helpful in our analysis of Case II.

\subsection{Case II: $I\neq\{0\}$}

Consider some $x\in X$ and the dimension group $(A,A^+_x,u)=(h\circ\sigma)(x)$.  Consider the quotient group $A/I$.  Theorem \ref{nature of theta} implies
$$a\in A^+_x\backslash\{0\}\text{ and }b\in I \implies a+b\in A^+_x\backslash\{0\},$$
since in this case $\widehat{(a+b)}=\widehat{a}+\widehat{b}=\widehat{a}$.  It is easy to see that $(A/I,C^+_x,v)$ is a simple dimension group, where $C^+_x=\{a+I\mid a\in A^+_x\}$ and $v=u+I$.  We check the Riesz Interpolation Property.  Consider $a,b,c,d\in A$ such that \mbox{$a+I,b+I\leq c+I,d+I$.}  Then $c-a+I, c-b+I, d-a+I, d-b+I\in C^+_x$.  This implies that \mbox{$c-a,c-b,d-a,d-b\in A^+_x$,} and so we may apply the Riesz Interpolation Property to obtain some $e\in A$ such that $a,b\leq e\leq c,d$.  Then $a+I,b+I\leq e+I\leq c+I,d+I$, and we are done.

%Let $\theta_x:A\rightarrow\Aff S_u(A,A^+_x)$ be the map $a\mapsto \hat{a}$ defined in the discussion preceding Theorem \ref{min-max}.  Since $I\neq\{0\}$, Corollary \ref{nature of theta} implies that $\theta_x$ is not injective. However, the image of $(A,A_x^+)$ under $\theta_x$ is a dimension group.  Let us call this group $(A/I,C_x^+)$, and define the distinguished order unit by $v=u+I$. For each $x\in X$,  $(A/I,C_x^+,v)$ is a dimension group of rank less than $n$ with a trivial group of infinitesimals.

So by Case I, we may assume that
there is a subset $X_1\subseteq X$ with $\mu(X_1)=1$ such that for
every $x,y\in X_1$, $(A/I,C_x^+,v)=(A/I,C_y^+,v)$.
%We will now show that for each $x,y\in X_1$, $(A,A_x^+,u)=(A,A_y^+,u)$.  To do this, we only need to show that given $x\in X$, we can recover  $A_x^+$ from $C_x^+$.
%Notice that for each state $p\in S_u(A,A^+)$, if $a\in I$, then $p(a)=0$.  Thus, if $a+I\in A/I$ and if $q$ is a state of $(A/I,C_x^+,v)$, then the map $p:A\rightarrow\mathbb{R}$ defined by $p(a)=q(a+I)$ is a state of $(A,A_x^+,u)$, and in fact every state of $(A,A_x^+,u)$ can be defined in this way.
Notice that if $p\in S_u(A,A^+_x)$, then $p^*(a+I)=p(a)$ defines a state $p^*\in S_v(A/I,C^+_x)$.  In fact, this defines a one-to-one correspondence between $S_u(A,A^+_x)$ and $S_v(A/I,C^+_x)$. Thus for $x,y\in X_1$, we have the following, where the last implication is due to Corollary \ref{extreme points determines A^+}:
\begin{align*}
(A/I,C_x^+,v)=(A/I,C_y^+,v)&\implies S_v((A/I,C_x^+))=S_v((A/I,C_y^+))\\
&\implies S_u(A,A_x^+)=S_u(A,A_y^+)\\
&\implies (A,A_x^+,u)=(A,A_y^+,u).
\end{align*}
And so $h\circ\sigma$ maps $X_1$ to a single dimension group.

\section{Application to LDA-groups}\label{section-application to LDA groups}
In this section, we refine Theorem \ref{LDA is more complex than all SDG_n} to show that the Borel complexity of countable simple thick unital LDA-groups is also strictly increasing with an appropriate notion of ``rank''.

%By defining the set of rank $n$ LDA groups to be the images of the various maps in Theorem \ref{LDA is more complex than all SDG_n}, one may prove a theorem like \ref{maintarget} for LDA groups.  However, there is not a currently known a \emph{natural} definition for the rank of an LDA group (or a Bratteli diagram).

\begin{definition}
Given a Bratteli diagram $(V,E,d)$ where $\displaystyle V=\bigsqcup_{n\in{\mathbb{N}}} V_n$, we define
$$\rank(V,E,d)=\liminf_{n\rightarrow\infty} |V_n|.$$
\end{definition}

However, Bratteli diagrams with different ranks may give rise to the same LDA-group (and the same dimension group). For example the rank $1$ diagram:

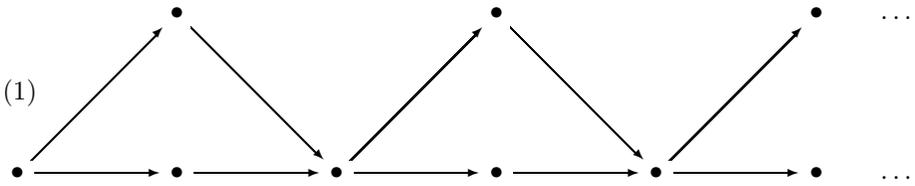
\begin{diagram}
&&\bullet&&&&\bullet&&&&\bullet&\dots\\
&\ruTo&&\rdTo&&\ruTo&&\rdTo&&\ruTo&&\\
\bullet&\rTo&\bullet&\rTo&\bullet&\rTo&\bullet&\rTo&\bullet&\rTo&\bullet&\dots\\
\end{diagram}

telescopes to the rank $2$ diagram:
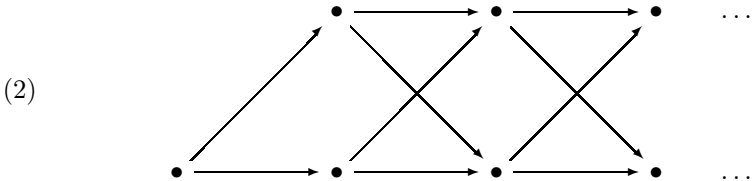
\begin{diagram}
&&\bullet&\rTo&\bullet&\rTo&\bullet&\dots\\
&\ruTo&&\ruTo\rdTo&&\ruTo\rdTo&&\\
\bullet&\rTo&\bullet&\rTo&\bullet&\rTo&\bullet&\dots\\
\end{diagram}
Thus we define:
\begin{definition}
Let $\mathcal{BD}_n$ denote the standard Borel space of simple Bratteli diagrams which are $\sim$-equivalent to a Bratteli diagram of rank at most $n$.  That is, we let
$$\mathcal{BD}_n=\{B\in\mathcal{BD}\mid \exists B'\sim B\text{ with }\rank B'\leq n\}.$$
Let $BD_n$ be the equivalence relation obtained by restricting $\sim$ to $\mathcal{BD}_n$.
\end{definition}

Notice that for any Bratteli diagram $B$, $\rank B\geq \rank(K_0(B))$.  Thus
\mbox{$BD_n\leq_B \left(\cong^+_n\right)$.}  Hence $BD_n$ is a Borel equivalence relation, and clearly $BD_n \leq_B BD_{n+1}$.

On the other hand, there are simple dimension groups whose rank is strictly less than that of each of the Bratteli diagrams which generate them. For example, Elliott \cite[2.7]{elliott2} has shown that the simple dimension group $A=\mathbb{Z}[\frac{1}{3}]\oplus\mathbb{Z}$ (here $\mathbb{Z}[\frac{1}{3}]$ denotes the triadic rationals) with positive cone $A^+=\{(a,b)\in A\mid a>0\}\cup \{(0,0)\}$ cannot be defined by a Bratteli diagram of rank less than $3$. However, we will find it convenient to ignore these types of dimension groups:

\begin{definition}
If a dimension group $(A,A^+,u)$ may be written as $K_0(B)$ for some Bratteli diagram $B$ where all the
maps $\varphi_n$ from the definition of $K_0(B)$ are one-to-one, then $(A,A^+,u)$ is said to be \emph{ultrasimplicial}.
\end{definition}

\begin{lemma}\label{ultrasimplicial does what we want}
If $(V,E,d)$ is a Bratteli diagram such that all the maps $\varphi_n$ are one-to-one, and $K_0(V,E,d)$ is a finite rank dimension group, then $\rank(V,E,d)=\rank(K_0(V,E,d))$, and there exists $N\geq 1$ such that $|V_n|=\rank(K_0(V,E,d))$, for all $n>N$.
\end{lemma}
\begin{proof}
Set $r=\rank(V,E)=\liminf_{n\rightarrow\infty} |V_n|$.  Let $N\geq 1$ be the least natural number such that $|V_N|=r$.  We claim that if $n> N$, then $|V_n|=|V_N|$.  Otherwise, either $|V_n|< |V_N|$ and then $\varphi_n\circ\ldots\circ\varphi_{N+1}$ is not injective, or else  $|V_n|> |V_N|$ and then there is some $m>n$ such that $|V_m|=|V_N|<|V_n|$ and $\varphi_m\circ\ldots\circ\varphi_{n+1}$ is not injective.

Next, the injectivity of the embeddings $\varphi_n$ implies that the natural basis of $\mathbb{Z}^{V_N}$ must be linearly independent in the limit.  Hence $\rank(K_0(V,E,d))\geq r$.
\end{proof}

We shall show that the dimension groups involved in the proof of Theorem \ref{maintarget} are all ultrasimplicial.

\begin{thm}\label{p-local implies ultrasimplicial}
Suppose $G$ is a $p$-local torsion-free abelian group of rank n, where $p>n$.  Then the dimension group $g_n(G)$
given by Lemma \ref{RQn to SDGn+1} is ultrasimplicial.
\end{thm}

Before we prove this, we show how this gives the analogue of Theorem \ref{maintarget} for simple Bratteli diagrams.

\begin{corollary}\label{main corollary}
For $n\geq 3$, $\BD_n <_B \BD_{n+1}$
\end{corollary}

\begin{proof}
Suppose that $f:\mathcal{BD}_{n+1} \rightarrow \mathcal{BD}_n$ is a Borel reduction from $\BD_{n+1}$ to $\BD_n$.  Let
\mbox{$g_n:R(\mathbb{Q}^n)\rightarrow SDG_{n+1}$} be the Borel reduction from $\cong_n$ to $\cong^+_{n+1}$ defined in Lemma \ref{RQn to SDGn+1}.
As in Section \ref{section-proof of main theorem}, we consider $X=X_n$, $\mu=\mu_n$, and $\sigma=\sigma_n$ from Theorem \ref{existence of probability space}.  If we pick $p>n$ when defining $X$, $\mu$, and $\sigma$, then Theorem \ref{p-local implies ultrasimplicial} says that every group in the image of $g_n\circ\sigma$ is ultrasimplicial.

Arguing as in the proof of Theorem \ref{LDA is more complex than all SDG_n}, we obtain a Borel reduction $$j:(g_n\circ\sigma)(X) \rightarrow \mathcal{BD}_{n+1}$$ from $\cong^+_{n+1}\upharpoonright_{(g_n\circ\sigma)(X)}$ to $BD_{n+1}$.  (Lemma \ref{ultrasimplicial does what we want} assures us that the rank of every Bratteli diagram in the range of this map is at most $n+1$.)  Next, the assignment $B\mapsto K_0(B)$ gives a Borel reduction
$$h:\mathcal{BD}_n\rightarrow \bigsqcup_{i\leq n} \mathcal{SDG}_i$$ from $BD_n$ to $\bigsqcup_{i\leq n}\cong^+_{i}$. Then the following composition is a Borel homomorphism from $E^X_{\SL_n(\mathbb{Z})}$ to $\bigsqcup_{i\leq n} \cong^+_i$:
$$X\stackrel{\sigma}{\rightarrow}R(\mathbb{Q}^n)\upharpoonright_{\sigma(X)}\stackrel{g_n}{\rightarrow} \mathcal{SDG}_{n+1}\upharpoonright_{(g_n\circ\sigma)(X)}\stackrel{j}{\rightarrow} \mathcal{BD}_{n+1} \stackrel{f}{\rightarrow} \mathcal{BD}_n \stackrel{h}{\rightarrow}\bigsqcup_{i\leq n} \mathcal{SDG}_i.$$

Clearly there exists a subset $X_1\subseteq X$ with $\mu(X_1)>0$ such that the above maps $X_1$ to $\mathcal{SDG}_k$ for some $k\leq n$.  Then by the ergodicity of $\mu$, $\mu(\SL_n(\mathbb{Z}).X_1)=1$.  Replacing $X$ by $\SL_n(\mathbb{Z}).X_1$, the analysis of Section \ref{section-proof of main theorem} again shows that there is a measure one subset of $X$ which maps to a single $\cong^+_k$-class.  This implies that $\sigma$ maps a measure one subset of $X$ to a single $\cong_n$-class, which is a contradiction.
\end{proof}

\begin{proof}[Proof of Theorem \ref{p-local implies ultrasimplicial}]

We will prove that $g_n(G)$ satisfies the following criteria for ultrasimpliciality:

\begin{lemma}{\cite{handelman}}
Let $(A,A^+,u)$ be a countable dimension group.  Then $(A,A^+,u)$ is ultrasimplicial if and only if for all
finite subsets $\{x_i\}^n_{i=1}$ of $A^+$,

$(\ast)$ there exists a finite subset $\{s_j\}^m_{j=1}$ of $A^+$ such that
\begin{enumerate}
\item  $\{s_j\}^m_{j=1}$ is rationally independent;
\item there exist $m_{ij}$ in $\mathbb{N}\cup\{0\}$ with $x_i=\sum m_{ij} s_j$, for all $i$.
\end{enumerate}
\end{lemma}

And we will use the following extension of the notions of $\gcd$ and $\lcm$ to the rationals.
\begin{definition}
Given a finite set of positive rational numbers $\{q_1,q_2,\ldots,q_n\}$, define $\gcd\{q_1,q_2,\ldots,q_n\}$ to be the greatest positive rational number $q$ such that for every $1\leq i\leq n$, $q_i=m_i q$ for some $m_i\in\mathbb{N}^+$.  Similarly let $\lcm\{q_1,q_2,\ldots,q_n\}$ be the least positive rational $r$ such that $1\leq i\leq n$, $r=m_i q_i$ for some $m_i\in\mathbb{N}^+$.
\end{definition}
Let $G$ be a $p$-local torsion-free abelian group of rank n, where $p>n+1$.  Let $(G\oplus\mathbb{Q},
(G\oplus\mathbb{Q})^+,(0,1))$ be the dimension group defined by $$(G\oplus\mathbb{Q})^+=\{ (h,q)\in
G\oplus\mathbb{Q} : q>0\}\cup \{ (0,0)\}.$$ Recall that $G\leq\mathbb{Q}^n$, and let
$\displaystyle\{x^i=(x_0^i,x_1^i,x_2^i,\ldots,x_{n-1}^i)\oplus(x_n^i)\}_{i\leq m}$ be a finite set of elements of
$(G\oplus\mathbb{Q})^+$. Let $\displaystyle y_k=\frac{1}{n+1}\gcd_{i\leq m}\{\left|x_k^i\right|\}$ for all \mbox{$0\leq k\leq n$.}
Next, for $0\leq j\leq n-1$, let $s_j=(0,0,\ldots,y_j,\ldots,0)\oplus(\frac{y_n}{n^{N}})$, where $y_j$ is in the
$j$-th slot, and $N\in\mathbb{N}$ is some constant determined below.  Finally, let
$$s_n=(-y_0,-y_1,\ldots,-y_j,\ldots,-y_{n-1})\oplus(\frac{y_n}{n^{N}}).$$

We claim that if $N$ is large enough, then $\{s_j\}_{i=0}^n$ fulfills $(\ast)$.  Clearly, the $\{s_j\}$ are rationally
independent.  Given $i\leq m$, we want to express $x^i$ as a sum of nonnegative integer multiples of the $s_j$.  Let $\displaystyle z=\lcm_{i,k} \{| x_k^i |,1\}$, and consider the sum
$$\displaystyle S = z s_n +\sum_{k=0}^{n-1} \left( \frac{x_k^i}{y_k} + z\right) s_k.$$
The coefficients $z$ and $\displaystyle \frac{x_k^i}{y_k} + z$ are both positive integers, and note that, for $0\leq k < n$,
$$S_k=z (-y_k)+\left(\frac{x^i_k}{y_k} +z \right) y_k = x^i_k.$$
We can then add some multiple $M^i$ of $\displaystyle \sum_{j=0}^n s_j$ to this sum
without changing the first $n$ coordinates.  So we just solve for $M^i$. We have that
$$x^i=z s_n+\sum_{k=0}^{n-1} \left(\frac{x_k^i}{y_k}+z\right)s_k + M^{i}\sum_{j=0}^n s_j$$
thus,
$$x_n^i=\frac{z y_n}{n^N}+\sum_{k=0}^{n-1}\left( \frac{x_k^i}{y_k}+z\right)  \frac{y_n}{n^N} + M^{i}(n+1)\frac{y_n}{n^N}$$
Then,
$$M^i=\frac{x_n^i-\sum_{k=0}^{n-1}\left( \frac{x_k^i}{y_k} +z \right)\frac{y_n}{n^N} - \frac{z y_n}{n^N}} {(n+1)\frac{y_n}{n^{N}}}
=\frac{x_n^i}{y_n}\frac{n^{N}}{n+1}-\left(\sum_{k=0}^{n-1}\frac{x_k^i}{y_k}\right)\frac{1}{n+1}-(n+1)\frac{z}{n+1}$$

By the definitions of $y_k$ and $z$, the righthand expression shows us that $M^i$ is an integer. Since $x^i_n >0$, the lefthand expression shows us that, if we choose $N$ large enough, then each of the (finitely many) $M^i$ can be made positive.
\end{proof}

Finally, we apply the above analysis to LDA-groups.

\begin{definition}
Given an LDA-group $G$, define
$$\rank(G)=\min\{\rank B\mid G(B)\cong G\}.$$
\end{definition}

\begin{definition}
For each $n\geq 1$, let $\mathcal{LDA}_n\subseteq\mathcal{LDA}$ be the standard Borel space of countable simple thick unital LDA-groups rank at most $n$.
Then let $\cong^{LDA}_n$ be the equivalence relation obtained by restricting $\cong_\mathcal{LDA}$ to $\mathcal{LDA}_n$.
\end{definition}

Theorem \ref{all three equivalent for thick unital} implies that the assignment $(V,E)\mapsto G(V,E)$ gives a Borel reduction from $BD_n$ to $\cong^{LDA}_n$, while the map described in the proof of Theorem \ref{LDA is more complex than all SDG_n} gives a Borel reduction from $\cong^{LDA}_n$ to $BD_n$.  Thus we have shown:

\begin{thm}
For each $n\geq 1$, $\left(\cong^{LDA}_n\right) \sim_B \left(BD_n\right)$.
\end{thm}

\begin{corollary}
For each $n\geq 3$, $\left(\cong^{LDA}_n\right) <_B \left(\cong^{LDA}_{n+1}\right)$.
\end{corollary}

Bratteli diagrams also characterize other naturally occurring structures, such as approximately finite-dimensional (AF) $C^*$-algebras and $AF$-relations on Cantor sets.

\begin{question}
For which other classes of structures that are described by Bratteli diagrams can we obtain a result similar to Theorem \ref{maintarget}?
\end{question}

Recall that $\cong_n$ is the isomorphism relation on the space of torsion free abelian groups of rank $n$.  Furthermore there is a universal countable Borel equivalence relation $E_\infty$.  That is if $E$ is countable Borel, then $E\leq_B E_\infty$.  In \cite{thomas-popa}, Thomas showed that $\left(\bigsqcup_{n\geq 1} \cong_n \right)<_B E_\infty$.
\begin{conjecture}
$\left(\bigsqcup_{n\geq 1} \cong_n^+ \right)<_B E_\infty$.
\end{conjecture}

It is natural to define the class of Bratteli diagrams of rank exactly $n$ as
$$\mathcal{BD}^*_n=\{B\in\mathcal{BD}\mid \sup\{\rank B' \text{ for }B'\sim B\}= n\},$$
and then to define $BD^*_n$ as $\sim\upharpoonright_{\mathcal{BD}^*_n}$.  It is easy to rewrite the proof of Corollary \ref{main corollary} to show that, for $n\geq 3$, $BD^*_{n+1}\nleq_B BD^*_n$.  However, the intuitively ``easy'' fact below is not currently known.
\begin{conjecture}
For $n\geq 1$, $\left(BD^*_{n}\right) \leq_B \left(BD^*_{n+1}\right)$.  Thus for $n\geq 3$, \mbox{$\left(BD^*_{n}\right) <_B \left(BD^*_{n+1}\right)$.}
\end{conjecture}

Also, given the results of Section \ref{section - two borel reductions and two small cases}, the following seems plausible:
\begin{conjecture}
For $n\geq 1$, $\left(\cong_n\right) \sim_B \left(\cong_{n+1}^+\right)$
\end{conjecture}

%\section{AF algebras}

%We first review some facts from the theory of approximately finite dimensional, or AF, $C^*$-algebras.  An arbitrary finite dimensional $C^*$-algebra takes the form, up to %isomorphism:
%$$\oplus_k M_{n_k},$$
%where $M_i$ denotes the full matrix algebra of $i\times i$ matrices.

%\begin{definition}
%A $C^*$-algebra is AF if it is the direct limit of a sequence of finite dimensional $C^*$-algebras:
%$$A=\underrightarrow{\lim}\; A_1 \stackrel{1}{\rightarrow}A_2\stackrel{2}{\rightarrow}\ldots A_i\stackrel{i}{\rightarrow}A_{i+1}\ldots,$$
%where each $A_i$ is a finite dimensional $C^*$-algebra
%\end{definition}

\end{document}